\documentclass[a4paper,11pt]{article}

\usepackage{amsfonts}
\usepackage[none]{hyphenat} 
\usepackage{cite}
\usepackage{latexsym}
\usepackage{amsthm}
\usepackage{fancyheadings}
\usepackage{amsmath}
\usepackage{graphicx}
\usepackage{float}
\usepackage{multirow}

\restylefloat{table}

\newtheorem{Exam}{Example}

\newtheorem{Lem}{Lemma}
\newtheorem{Con}{Conjecture}
\newtheorem{The}{Theorem}

\newtheorem{Pro}{Proposition}
\theoremstyle{definition}
\newtheorem{De}{Definition}
\newtheorem{Rem}{Remark}

\newcommand{\bsb}[1]{\boldsymbol{#1}}
\def\st{{\mathsf{st}}}
\def\m{{\mathsf{m}}}
\def\len{{\mathsf{len}}}
\def\asc{{\mathsf{asc}}}
\def\lv{{\mathsf{lv}}}
\def\desc{{\mathsf{des}}}
\def\lev{{\mathsf{lev}}}

\DeclareMathOperator{\suc}{succ}
\DeclareMathOperator{\tildesuc}{\widetilde{succ}}
\DeclareMathOperator{\last}{last}
\DeclareMathOperator{\first}{first}

\DeclareMathOperator{\Atilde}{\mathcal{\widetilde A}}
\DeclareMathOperator{\SEtilde}{\mathcal{\widetilde {SE}}}

\DeclareMathOperator{\Rtilde}{\mathcal{\widetilde R}}
\DeclareMathOperator{\Stilde}{\mathcal{\widetilde S}}

\DeclareMathOperator{\Xtilde}{\mathcal{\widetilde X}}
\DeclareMathOperator{\Ytilde}{\mathcal{\widetilde Y}}

\begin{document}

\title{Two Reflected Gray Code based orders \\on some restricted growth
sequences}

\author{
\begin{tabular}{cc}
Ahmad {\sc Sabri}&Vincent {\sc Vajnovszki}\\
{\small LE2I, Universit\'e de Bourgogne}&{\small LE2I, Universit\'e de Bourgogne}\\
{\small BP 47870, 21078 Dijon Cedex, France}&{\small BP 47870, 21078 Dijon Cedex, France}\\
{\small \tt ahmad.sabri@u-bourgogne.fr}&{\small \tt vvajnov@u-bourgogne.fr}\\
{\small Dept. of Informatics, Gunadarma University}&\\
{\small Depok 16424, Indonesia}&\\ 
{\small \tt sabri@staff.gunadarma.ac.id}&\\
\end{tabular}
}

\maketitle

\begin{abstract}
We consider two order relations: that induced by the $m$-ary reflected 
Gray code and a suffix partitioned variation of it.
We show that both of them when applied to some sets of restricted growth 
sequences still yield Gray codes.
These sets of sequences are: subexcedant or ascent sequences, restricted growth 
functions, and staircase words.
In each case we give efficient 
exhaustive generating algorithms and compare the obtained results.

\end{abstract}

\section{Introduction and motivations}

The term `Gray code' was taken from Frank Gray, who patented {\it
Binary Reflected Gray Code} (BRGC) in 1953 \cite{Gray}. The concept of BRGC is
extended to {\it Reflected Gray Code} (RGC), to accommodate $m$-tuples (sequence),
with $m>2$ \cite{Er}. In these Gray codes, successive sequence
differ in a single position, and by $+1$ or $-1$ in this position.
More generally, if a list of sequences is such that the Hamming distance between
successive sequences is upper bounded by a constant $d$, then the list is said 
a $d$-{\it Gray code}. So in particular, BRGC and RGC are 1-Gray codes. In
addition, if the positions where the successive sequences differ are adjacent, then we say that the list is a $d$-{\it adjacent} Gray code.

For long time, the design of Gray codes for 
combinatorial classes and their corresponding generating
algorithms was an ad-hoc task, that is, done case by case according
to the class under consideration.
Recently, general techniques which fit to
large classes of combinatorial objects were developed and used.
Among them are, for example, the ECO-method (initiated in 
Gray code context in \cite{Bern},
see also \cite{Vaj_2010}), prefix rotations
(yielding bubble languages, see \cite{Rus_Will,Rus_Saw_Will} and references therein), or
Reflected Gray Code based order relations; this last technique was
used implicitly, for example in \cite{Kling,Walsh_2000}, and developed
systematically as a general method in 
\cite{Bar_Vaj,Vaj_F,Vaj_L,Vajnov2,Vaj_2010,VajnovVernay}.
The results presented in this paper are in the light
of this last direction. More precisely, we show that
two order relations induced by Reflected Gray Code and its variation
also give Gray codes for some classes of restricted growth sequences defined by means of statistics. These classes are:
subexcedant or ascent sequences, restricted growth functions, and staircase words.
We give efficient (CAT) generating algorithm for each obtained Gray code.

\section{Preliminaries}
\label{sect:preliminaries}
\subsection{Gray code orders}

Let $G_n(m)$ be the set of length $n$ $m$-ary sequences $s_1s_2\ldots s_n$ with
$s_i\in\{0,1,\ldots, m-1\}$; clearly, $G_n(m)$ is the product set
$\{0,1,\ldots,m-1\}^n$.
The {\it Reflected Gray Code} ({\it RGC} for short) for the set
$G_n(m)$, denoted by ${\mathcal G}_n(m)$, is
the natural extension of the Binary Reflected Gray Code to this set.
The list ${\cal G}_n(m)$ is defined recursively by the following relation
\cite{Er}:
\begin{equation}
\label{eq:RGC}
{\cal G}_n(m)=\left\{ \begin {array}{ccc}
\epsilon   &  {\rm if} & n=0,  \\
0            {\cal G}_{n-1}(m),\;
1 \overline {{\cal G}_{n-1}(m)},\;
2            {\cal G}_{n-1}(m),\; \dots ,\;
(m-1)          {\cal G}_{n-1}'(m)
& {\rm if} & n>0,
\end {array}
\right.
\end{equation}
where $\epsilon$ is the empty sequence, $\overline {{\cal G}_{n-1}(m)}$ is the reverse of  ${\cal G}_{n-1}(m)$,
and ${\cal G}_{n-1}'(m)$ is ${\cal G}_{n-1}(m)$
or $\overline {{\cal G}_{n-1}(m)}$ according to $m$ is odd or even.

In ${\cal G}_n(m)$, two successive sequences differ in a single position and by
$+1$ or $-1$ in this position. A list for a set of sequences induces an order
relation to this set, and we give two order relations induced by the RGC and its
variation, namely RGC order \cite{Vaj_F} and Co-RGC order.

We adopt the convention that lower case bold letters represent tuples,
for example: $\bsb s=s_1s_2\ldots s_n$, $\bsb a=a_1a_2\ldots a_k$,
$\bsb b=b_{k+1}b_{k+2}\ldots b_n$.

\begin{De}
The {\it Reflected Gray Code order} $\prec$ on $G_n(m)$ is defined as:
$\bsb{s}=s_1s_2\ldots s_n$ is less than $\bsb{t}=t_1t_2\ldots t_n$,
denoted by $\bsb{s}\prec\bsb{t}$, if either
\begin{itemize}
\item
  $\sum_{i=1}^{k-1} s_i$ is
  even and $s_k<t_k$, or
\item
  $\sum_{i=1}^{k-1} s_i$ is
  odd and $s_k>t_k$,
\end{itemize}
where $k$ is the leftmost position where ${\bsb s}$ and ${\bsb t}$ differ.
\end{De}
It is easy to see that ${\cal G}_n(m)$ defined in relation (\ref{eq:RGC}) lists
sequences in $G_n(m)$ in $\prec$ order.

Now we give a variation of ${\cal G}_n(m)$.
Let $s_1s_2\ldots s_n$ be a sequence in 
$G_n(m)$. The complement of $s_i$, $1\leq i\leq n$,
is 
$$(m-1-s_i),$$
and the reverse of $s_1s_2\ldots s_n$ is
$$s_ns_{n-1}\ldots s_1.$$
Let ${\widetilde{\cal G}}_n(m)$ be the list obtained by transforming each sequence $\bsb s$ in ${\cal G}_n(m)$ as follows:
\begin{itemize}
\item complementing each digit in $\bsb s$ if $m$ is even, 
      or complementing only digits in odd positions if $m$ is odd, then
\item reversing the obtained sequence.
\end{itemize}
Clearly, ${\widetilde{\cal G}}_n(m)$ is also a Gray code for 
$G_n(m)$, and the sequences therein are listed in Co-Reflected Gray Code order,
as defined formally below.

\begin{De}
\label{De:CoRGC}
The {\it Co-Reflected Gray Code order} ${\prec}_c$ on $G_n(m)$ is defined
as:\\
$\bsb{s}=s_1s_2\ldots s_n$ is less than $\bsb{t}=t_1t_2\ldots t_n$, denoted by
$\bsb{s}{\prec}_c\bsb{t}$, if either
\begin{itemize}
\item $\sum_{i=k+1}^{n}s_i+(n-k)$ is even and $s_k>t_k$, or
\item $\sum_{i=k+1}^{n}s_i+(n-k)$ is odd and $s_k<t_k$,
\end{itemize}
where $k$ is the rightmost position where ${\bsb s}$ and ${\bsb t}$ differ.
\end{De}

Although this definition sounds somewhat arbitrary, as we will see in Section \ref{sec:coRGC}, it turns out that $\prec_c$ order gives suffix partitioned Gray codes for some sets of restricted growth sequences.
Obviously, the restriction of ${\cal G}_n(m)$ (resp. ${\widetilde{\cal G}}_n(m)$) to a set of sequences is simply the list of sequences in the set listed in $\prec$ (resp. $\prec_c$) order.

\subsection{Restricted growth sequences defined by means of statistics}

Through this paper we consider sequences over non-negative integers. A {\it statistic} on a set of sequences is an association of an integer to each
sequence in the set. For a sequence $s_1s_2\ldots s_n$, its length minus one, numbers of ascents/levels/descents, maximal value, and last value are classical examples of statistics. They are defined as follows, see also \cite{Man_Vaj}:

\begin{itemize}
\item{$\len(s_1s_2\ldots s_n)=n-1$};
\item{$\asc(s_1s_2\ldots s_{n})={\rm card}\{i\,|\,1\le i< n {\rm \ and\ }
s_i<s_{i+1}\}$; }
\item{$\lev(s_1s_2\ldots s_n)={\rm card}\{i\,|\,1\le i< n {\rm \ and\ }
s_i=s_{i+1}\}$; }
\item{$\desc(s_1s_2\ldots s_n)={\rm card}\{i\,|\,1\le i< n {\rm \ and\ }
s_i>s_{i+1}\}$; }
\item{$\m(s_1s_2\ldots s_n)=\max\{s_1,s_2,\ldots,s_n\}$; }
\item{$\lv(s_1s_2\ldots s_n)=s_n$. }

\end{itemize}

\noindent
If $\st$ is one of the statistics $\len$, $\asc$, $\m$, and $\lv$, then $\st$ satisfy the following:
\begin{equation}
\label{prop1}
\st(s_1s_2\ldots s_n)\leq n-1,
\end{equation}
and
\begin{equation}
\label{prop2}
{\rm if}\ s_n=\st(s_1s_2\ldots s_{n-1})+1, {\rm then}\ s_n=\st(s_1s_2\ldots s_{n-1}s_n).
\end{equation}
\noindent
On the contrary, the statistics $\lev$ and $\desc$ do not satisfy 
relation (\ref{prop2}). Accordingly, through this paper we will consider only the four statistics above.
However, as we will point out, some of the results presented here are also true for arbitrary statistics satisfying relations (\ref{prop1}) and (\ref{prop2}).

\begin{De}
\label{def:RGS}
For a given statistic $\st$, 
an $\st$-{\it restricted growth sequence} $s_1s_2\ldots s_n$ is a sequence with
$s_1=0$ and
\begin{equation}
\label{Un_Def}
0\leq s_{k+1}\leq \st(s_1s_2\ldots s_k)+1 {\rm \ for \ } 1\le k<n,
\end{equation}
and the set of $\st$-restricted growth sequences is the set
of all sequences  $s_1s_2\ldots s_n$ satisfying relation (\ref{Un_Def}).
\end{De}

From this definition, it follows that any prefix of an 
$\st$-restricted growth sequence
is also (a shorter) $\st$-restricted growth sequence.

\begin{Rem}
\label{prop}
If $\st$ is a statistic satisfying relations (\ref{prop1}) and (\ref{prop2}) above, then
\begin{enumerate}
\item $\max\{\st(s_1s_2\ldots s_n)\,|\,s_1s_2\ldots s_n {\rm \,is\ an\ }
\st$-restricted growth sequence$\}=n-1$;
\item if $s_1s_2\ldots s_n$ is an $\st$-restricted growth sequence,
then for any $k$, $1\leq k<n$,
$s_{k+1}=\st(s_1s_2\ldots s_k)+1$
implies $s_{k+1}=\st(s_1s_2\ldots s_k s_{k+1})$.

\end{enumerate}
\end{Rem}

The sets of $\st$-restricted growth sequences, where $\st$ is one of the statistics $\len$, $\asc$, $\m$, and $\lv$, are defined below.

\begin{De}
\
\begin{itemize}
\item{The set $SE_n$ of {\it subexcedant sequences} of length $n$ is defined as:
$$SE_n=\{ s_1s_2\ldots s_n \, | \,  s_1=0 {\rm \ and \ }
0\leq s_{k+1}\leq \len(s_1s_2\ldots s_k)+1 {\rm \ for \ } 1\le k<n \};$$}
\item{The set $A_n$ of {\it ascent sequences} of length $n$ is defined as:
$$A_n=\{ s_1s_2\ldots s_n \, | \,  s_1=0 {\rm \ and \ }
0\leq s_{k+1}\leq \asc(s_1s_2\ldots s_{k})+1 {\rm \ for \ }1\le k<n \};$$}
\item{The set $R_n$ of {\it restricted growth functions} of length $n$ is
defined as:
$$R_n=\{ s_1s_2\ldots s_n \, | \,  s_1=0 {\rm \ and \ }
0\leq s_{k+1}\leq \m(s_1s_2\ldots s_{k})+1 {\rm \ for \ }1\le k<n\};$$}
\item{The set $S_n$ of {\it staircase words} of length $n$ is defined as:
$$S_n=\{ s_1s_2\ldots s_n \, | \,  s_1=0 {\rm \ and \ }
0\leq s_{k+1}\leq \lv(s_1s_2\ldots s_{k})+1 {\rm \ for \ }1\le k<n \}.$$}
\end{itemize}
\end{De}

\medskip
Notice that alternatively, $SE_n=\{0\} \times \{0,1\} \times \ldots \times
\{0,1,\ldots,n-1\}$.

\begin{Rem}
\label{re:comment}
$S_n\subset R_n\subset A_n\subset SE_n\subset G_n(n)$.
\end{Rem}

\noindent Below we give examples to illustrate Remark \ref{re:comment}.

\begin{Exam}$ $
\begin{itemize}
\item If $\bsb s=010145$, then $\bsb s\in SE_6$, but $\bsb s\notin A_6$,
$\bsb s\notin R_6$ and $\bsb s\notin S_6$.
\item If $\bsb s=010103$, then $\bsb s\in SE_6$, $\bsb s\in A_6$,
but $\bsb s\notin R_6$ and $\bsb s\notin S_6$.
\item If $\bsb s=010102$, then $\bsb s\in SE_6$, $\bsb s\in A_6$, and $\bsb s\in R_6$,
but $\bsb s\notin S_6$.
\item If $\bsb s=010101$, then $\bsb s\in SE_6$, $\bsb s\in A_6$, $\bsb s\in R_6$
and $\bsb s\in S_6$.
\end{itemize}
\end{Exam}

\noindent 
See  Table \ref{tab:example_A5} for the sets $S_5$, $R_5$, and $A_5$ listed in $\prec$ order, and Table \ref{tab:example_Atilde5} for the same sets listed in $\prec_c$ order.

\noindent We denote by
\begin{itemize}
\item $\mathcal X_n$ the list for the set $X_n$ in $\prec$ order, and  by $\Xtilde_n$ that in $\prec_c$ order;
\item{$\suc_{X}(\bsb s)$, $\bsb s\in X_n$, the successor of $\bsb s$ in the set
$X_n$ listed in $\prec$ order; that is, the smallest sequence in $X_n$ larger
than $\bsb s$ with respect to $\prec$ order;}
\item{$\widetilde{\suc}_{X}(\bsb s)$ the counterpart of $\suc_{X}(\bsb s)$
with respect to ${\prec}_c$ order;   }
\item{$\first(\mathcal L)$ the first sequence in the list $\mathcal L$;}
\item{$\last(\mathcal L)$ the last sequence in the list $\mathcal L$.}

\end{itemize}

\subsection{Constant Amortized Time algorithms and principle}
\label{subsect:CAT}
An exhaustive generating algorithm is said to run in {\it constant amortized
time} ({\it CAT} for short) if the total amount of computation is proportional
to the number of generated objects. And so, a CAT algorithm can be considered an
efficient algorithm.

Ruskey and van Baronaigien \cite{RuskBar1} introduced three CAT properties, and
proved that if a recursive generating procedure satisfies them, then it
runs in constant amortized time (see also \cite{Rus}). They called this general technique to prove the
efficiency of a generating algorithm as {\it CAT principle}, and the involved
properties are:
\begin{enumerate}
\item Every call of the procedure results in the output of at least one object;
\item Excluding the computation done by the recursive calls, the amount of
      computation of any call is proportional to the {\it degree of the call}, that
      is, the number of call initiated by the current call;
\item The number of calls of degree one, if any, is $O(N)$, where $N$ is the number of
      generated objects.
\end{enumerate}

\noindent All the generating algorithms we present in this paper satisfy these
three desiderata, and so they are efficient.
\section{The Reflected Gray Code order for the sets $SE_n$, $A_n$, $R_n$, and $S_n$}
\label{sec:RGC}
\subsection{The bound of Hamming distance between successive sequences in the lists $\mathcal {SE}_n$,
$\mathcal A_n$, $\mathcal R_n$, and~$\mathcal S_n$}
Here we will show that the Hamming distance between two successive sequences in each of the mentioned lists 
is upper bounded by a constant, and so the lists are Gray codes.

Without another specification, $X_n$ generically denotes one of the
sets $SE_n$, $A_n$, $R_n$, or $S_n$; and
$\mathcal X_n$ denotes its corresponding list in $\prec$ order,
that is, one of the lists $\mathcal{SE}_n$, ${\mathcal A}_n$, ${\mathcal R}_n$, 
or ${\mathcal S}_n$.
Later in this section, Theorem~\ref{The_Gray_A_others} and Proposition \ref{G_SE}
state that, in each case, the set $X_n$ listed in $\prec$ order 
yields a 
Gray code.


\begin{Lem}
\label{lem:differ-is-one}
If $\bsb s=s_1 s_2\ldots s_n$ and $\bsb t=t_1 t_2\ldots t_n$ are two sequences
in $X_n$ with $\bsb t=\suc_X(\bsb s)$ and $k$ is the leftmost position
where they differ, then $s_k=t_k+1$ or $s_k=t_k-1$.
\end{Lem}

\proof
Let $\bsb t=\suc_X(\bsb s)$ and $k$ be the leftmost position where they differ.
Let us suppose that $s_k<t_k$ and $s_k\neq t_k-1$
(the case $s_k>t_k$ and $s_k\neq t_k+1$
 being similar).\\
It is easy to check that
$$\bsb u=s_1s_2\ldots s_{k-1}(s_k+1)0\ldots 0$$
belongs to $X_n$, and considering the definition of $\prec$ order relation,
it follows that
$\bsb s\prec\bsb u\prec \bsb t$, which is in contradiction
with $\bsb t=\suc_X(\bsb s)$, and the statement holds. 
\endproof

If ${\bsb a}=a_1a_2\ldots a_k\in X_k$, then for any $n>k$,
${\bsb a}$ is the prefix of at least one sequence in $X_n$,
and we denote by ${\bsb a}\,|\,\mathcal X_n$
the sublist of $\mathcal X_n$ of all sequences having the prefix ${\bsb a}$.
Clearly, a list in $\prec$ order for a set of sequences
is a {\it prefix partitioned} list (all sequences with same prefix are
contiguous), and  for any
${\bsb a}\in X_k$ and $n>k$, it follows that ${\bsb a}\,|\,\mathcal X_n$
is a contiguous sublist of $\mathcal X_n$.

For a given $\bsb a\in X_k$, the set of all $x$ such that ${\bsb a}x\in X_{k+1}$ is called the {\it defining set of the prefix} ${\bsb a}$, and obviously
${\bsb a}x$ is also a prefix of some sequences in $X_n$, for any $n>k$.
We denote by
\begin{equation}
\label{de_M}
\omega_X(\bsb a)=\max\{x\,|\,{\bsb a}x \in X_{k+1} \}
\end{equation}
the largest value in the defining set of ${\bsb a}$.
And if we denote $\omega_X(\bsb a)$ by $M$, then by 
Remark \ref{prop} we have
\begin{equation}
\nonumber
\label{eq:increment}
\begin{array}{rcl}
M	&=&\st(\bsb a)+1\\
     &=&\st(\bsb aM).
\end{array}
\end{equation}

\noindent And consequently,
\begin{equation}
\label{eq:increment2}
\begin{array}{rcl}
\omega_X(\bsb aM)&=&\st(\bsb aM)+1\\
	&=&M+1.\\
\end{array}
\end{equation}

The next proposition gives the pattern of ${\bsb s} \in X_n$, if $\bsb
s=\last({\bsb a}\,|\,\mathcal X_n)$ or
$\bsb s=\first({\bsb a}\,|\,\mathcal X_n)$.

\begin{Pro}
\label{pro:RGClast}
Let $k<n$ and ${\bsb a}=a_1a_2\ldots a_k\in X_k$.
If $\bsb s=\last({\bsb a}\,|\,\mathcal X_n)$,
then the pattern of $\bsb s$ is given by:
\begin{itemize}
\item if $\sum_{i=1}^{k} a_i$ is odd, then $\bsb s={\bsb a}0\ldots 0$;    
\item if $\sum_{i=1}^{k} a_i$ is even and $M$ is odd, then
      $\bsb s={\bsb a}M0\ldots 0 $;
\item if $\sum_{i=1}^{k} a_i$ is even and $M$ is even, then
      $\bsb s={\bsb a}M (M+1) 0\ldots 0$;
\end{itemize}
where $M$ denotes $\omega_X(\bsb a)$.

\noindent 
Similar results hold for $\bsb s=\first({\bsb a}\,|\,\mathcal X_n)$
by replacing `odd' by `even', and {\em vice versa}, for the parity of $\sum_{i=1}^{k} a_i$.
\end{Pro}

\proof
\noindent
Let $\bsb s=a_1a_2\ldots a_ks_{k+1}\ldots s_n=\last({\bsb a}\,|\,\mathcal X_n)$.\\
If $\sum_{i=1}^{k} a_i$ is odd, then by considering the definition of 
$\prec$ order, it follows that $s_{k+1}$ is the smallest value in the
defining set of $\bsb a$, and so $s_{k+1}=0$, and finally
$\bsb s={\bsb a}0\ldots 0$, and the first point holds.

\noindent
Now let us suppose that $\sum_{i=1}^{k} a_i$ is even. In this case
$s_{k+1}$ equals $\omega_X(\bsb a)=M$, the largest value in the defining set of 
$\bsb a$.
When in addition $M$ is odd, so is the summation 
of $\bsb a M$, the length $k+1$ prefix of $\bsb s$, and thus 
$\bsb s={\bsb a}M0\ldots 0 $, and the second point holds.\\
Finally, when  $M$ is even, then $s_{k+2}$ is the largest value 
in the defining set of $\bsb a M$, which by relation (\ref{eq:increment}) is $M+1$. 
In this case $M+1$ is odd, and thus $\bsb s={\bsb a}M (M+1) 0\ldots 0$,
and the last point holds.\\
The proof for the case  $\bsb s=\first({\bsb a}\,|\,\mathcal X_n)$ is similar.
\endproof

By Proposition \ref{pro:RGClast} above, we have the following:

\begin{The}
\label{The_Gray_A_others}
The lists $\mathcal A_n$, $\mathcal R_n$ and $\mathcal S_n$ are
$3$-adjacent Gray codes.
\end{The}
\proof
Let $\mathcal X_n$ be one of the lists $\mathcal A_n$, $\mathcal R_n$ or 
$\mathcal S_n$, and $\bsb t=\suc_X(\bsb s)$.
Let $k$ be the leftmost position where ${\bsb s}$ and ${\bsb t}$
differ, and let us denote by $\bsb a$ the length $k$ prefix
of ${\bsb s}$ and $\bsb a'$ that of ${\bsb t}$;
so, $\bsb s=\last({\bsb a}\,|\,\mathcal X_n)$ and
$\bsb t=\first({\bsb a'}\,|\,\mathcal X_n)$. 
If $k+3\leq n$, then by Proposition \ref{pro:RGClast},
it follows that 
$s_{k+3}=s_{k+4}=\cdots=s_n=0$ and $t_{k+3}=t_{k+4}=\cdots=t_n=0$. 
So $\bsb s$ and $\bsb t$ differ only in position $k$, and possibly 
in position $k+1$ and in position $k+2$.

\noindent
Now we show the adjacency, that is, if $k+2\leq n$ and 
$s_{k+1}=t_{k+1}$ implies
$s_{k+2}=t_{k+2}$. If $s_{k+1}=t_{k+1}$,  
by Lemma \ref{lem:differ-is-one}, it follows that the summation of the length $k$
prefix of $\bsb s$ and that of $\bsb t$ have different parity,  
and two cases can occur:\\
$\bullet$ $s_{k+1}=t_{k+1}=0$, and by Proposition \ref{pro:RGClast}, it follows that $s_{k+2}=t_{k+2}=0$; or \\
$\bullet$ $s_{k+1}=t_{k+1}\neq 0$, and thus
$s_{k+1}=t_{k+1}=\omega({\bsb a})=\omega({\bsb a}')$. In this case,
$\omega({\bsb a})$ either is odd and so 
$s_{k+2}=t_{k+2}=0$, or is even
and so $s_{k+2}=t_{k+2}=\omega({\bsb a})+1$.

\noindent
In both cases, $s_{k+2}=t_{k+2}$.
\endproof

It is well known that the restriction of ${\cal G}_n(m)$ defined in relation 
(\ref{eq:RGC}) to any product space remains a $1$-Gray code, see for example 
\cite{VajnovVernay}.
In particular, for $SE_n=\{0\} \times \{0,1\} \times \ldots \times
\{0,1,\ldots,n-1\}$ we have the next proposition.
Its proof is simply based on Lemma \ref{lem:differ-is-one}, Proposition~\ref{pro:RGClast}, 
and on the additional remark: for any $\bsb a\in SE_k$, $k<n$, it follows that
$\omega_{SE}(\bsb a)=k$.

\begin{Pro}
\label{G_SE}
The list $\mathcal{SE}_n$ is $1$-Gray code.
\end{Pro}

It is worth to mention that for any statistic $\st$ satisfying relations (\ref{prop1}) and (\ref{prop2}),  the list in $\prec$ order for the set of $\st$-restricted growth sequences of length $n$ is an at most  $3$-Gray code.\\
Actually, the lists $\mathcal{SE}_n$,  $\mathcal A_n$, $\mathcal R_n$, and $\mathcal S_n$ are {\it circular} Gray codes, that is, the last and the first sequences in the list differ in the same way. Indeed, by the definition of $\prec$ order, it follows that:
\begin{itemize}
\item $\first(\mathcal X_n)=000\ldots 0$;
\item $\last(\mathcal X_n)=010\ldots 0$;
\end{itemize}
where $\mathcal X_n$ is one of the list $\mathcal{SE}_n$,  $\mathcal A_n$, $\mathcal R_n$, or $\mathcal S_n$.

\subsection{Generating  algorithms for the lists $\mathcal{SE}_n$, $\mathcal A_n$, $\mathcal R_n$, and $\mathcal S_n$}


\noindent Procedure ${\tt Gen1}$  in Figure \ref{fig:gen1} is a general 
procedure generating exhaustively the list of $\st$-restricted growth 
sequences, where $\st$ is a statistic satisfying 
relations (\ref{prop1}) and (\ref{prop2}).
According to particular instances of the 
function ${\tt Omega\_}X$ called by it
(and so, of the statistic $\st$),
${\tt Gen1}$ produces specific $\st$-restricted growth 
sequences, and in particular the lists $\mathcal {SE}_n$,
$\mathcal A_n$, $\mathcal R_n$, and $\mathcal S_n$.
From the length one sequence $0$, ${\tt Gen1}$ constructs
recursively increasing length $\st$-restricted growth 
sequences: for a given prefix $s_1s_2\ldots s_k$
it produces all prefixes $s_1s_2\ldots s_ki$, with $i$
covering (in increasing or decreasing order) the defining set of $s_1s_2\ldots s_k$;
and eventually all length $n$ $\st$-restricted growth 
sequences. It has the following parameters:
\begin{itemize}
\item $k$, the position in the sequence ${\bsb s}$
      which is updated by the current call;
\item $x$, belongs to the defining set of $s_1s_2\ldots s_{k-1}$, 
      and is the value to be assigned to $s_k$;
\item $dir$, the direction (ascending for $dir\mod 2=0$ and descending
      for $dir\mod 2=1$) to cover the defining set of $s_1s_2\ldots s_{k-1}$;
\item $v$, the value of the statistic of the prefix $s_1s_2\ldots s_{k-1}$
      from which the value of the statistic of 
      the current prefix $s_1s_2\ldots s_k$ is computed. Remark that $v=\omega_X(s_1s_2\ldots s_{k-1})-1$.
\end{itemize}
Function {\tt Omega\_}$X$ computes $\omega_X(s_1s_2\ldots s_k)$ (see relation (\ref{de_M})), and 
the main call is ${\tt Gen1}(1,0,0,0)$. 

\begin{figure}[!h]
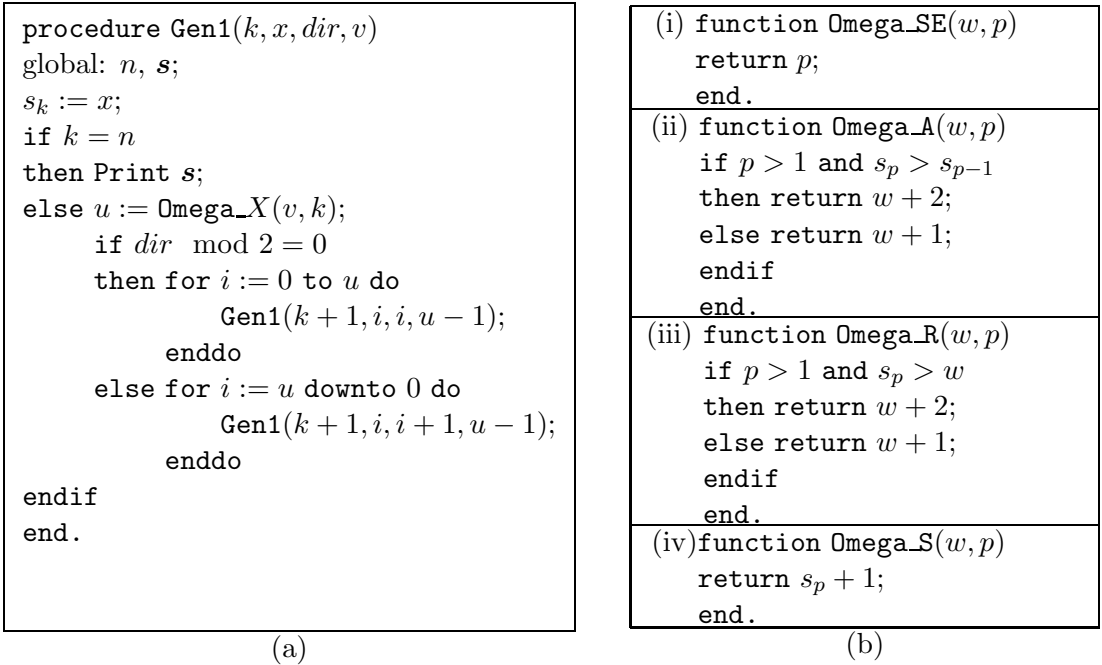

\centering
\begin{tabular}{c}
\fbox{
\begin{minipage}[t][8.1cm]{5cm}
\begin{tabbing}
{\tt procedure} ${\tt Gen1}(k,x,dir,v)$\\
global: $n$, ${\bsb s}$;\\ 
$s_k:=x$;\\
{\tt if} $k=n$ \\
{\tt then} {\tt Print} ${\bsb s}$;\\
{\tt else} \=$u:={\tt Omega\_}X(v,k)$;\\
	\>{\tt if} $dir \mod 2=0$\\
	\>{\tt then} \={\tt for} \=$i:=0$ {\tt to} $u$ {\tt do}\\
	\>		\>	         \>${\tt Gen1}(k+1,i,i,u-1)$;\\
	\>		\>{\tt enddo}\\
	\>{\tt else} \={\tt for} \=$i:=u$ {\tt downto} $0$ {\tt do}\\
	\>		\>	         \>${\tt Gen1}(k+1,i,i+1,u-1)$;\\
	\>		\>{\tt enddo}\\

{\tt endif}\\
{\tt end.}
\end{tabbing}
\end{minipage}
}\\
(a)
\end{tabular}
\quad
\begin{tabular}{c}\hline

	\multicolumn{1}{|c|}{

	(i) \begin{minipage}[t][1.3cm]{5cm}
		{\tt function} {\tt Omega\_SE}$(w,p)$\\
		{\tt return} $p$;\\		
		{\tt end.}	\\
	\end{minipage}
	}

\\ \hline
	\multicolumn{1}{|c|}{

	 (ii) \begin{minipage}[t][2.7cm]{5cm}
		{\tt function} {\tt Omega\_A}$(w,p)$\\
		{\tt if} $p>1$ {\tt and} $s_p>s_{p-1}$\\
		{\tt then} {\tt return} $w+2$;\\
		{\tt else} {\tt return} $w+1$;\\
		{\tt endif}\\		
		{\tt end.}		
	\end{minipage}
	}
\\	\hline
	\multicolumn{1}{|c|}{

	(iii) \begin{minipage}[t][2.7cm]{5cm}
		{\tt function} {\tt Omega\_R}$(w,p)$\\
		{\tt if} $p>1$ {\tt and} $s_p>w$\\
		{\tt then} {\tt return} $w+2$;\\
		{\tt else} {\tt return} $w+1$;\\
		{\tt endif}\\		
		{\tt end.}
	\end{minipage}
	}

\\ \hline
	\multicolumn{1}{|c|}{

	(iv)\begin{minipage}[t][1.3cm]{5cm}
		{\tt function} {\tt Omega\_S}$(w,p)$\\
		{\tt return} $s_p+1$;\\		
		{\tt end.}
		\phantom{text text text text text text text}
	\end{minipage}
	}

\\ \hline

(b)
\end{tabular}

\caption{(a) Algorithm {\tt Gen1}, generating the list $\mathcal X_n$; (b) Particular function {\tt Omega\_}$X$ called by {\tt Gen1}, and returning the value for $\omega_{X}(s_1s_2\ldots s_k)$, if $X_n$ is one of the sets:  (i) $SE_n$, (ii) $A_n$, (iii) $R_n$, and (iv) $S_n$. 
}
\label{fig:gen1}
\end{figure}
 
In ${\tt Gen1}$ the amount of computation of each call is proportional with 
the degree of the call, and there are no degree one calls, and so it 
satisfies the CAT principle
stated at the end of Section \ref{sect:preliminaries}, and so it is an efficient generating 
algorithm.
The computational tree of ${\tt Gen1}$ producing the list $\mathcal A_4$
is given in Figure \ref{fig:tree_pref}. Each node at level $k$,
$1\leq k\leq 4$, represents prefixes $s_1s_2\ldots s_k$, and leaves 
sequences in  $\mathcal A_4$.

\begin{figure}[h]
\begin{center}\includegraphics[width=8.5cm]{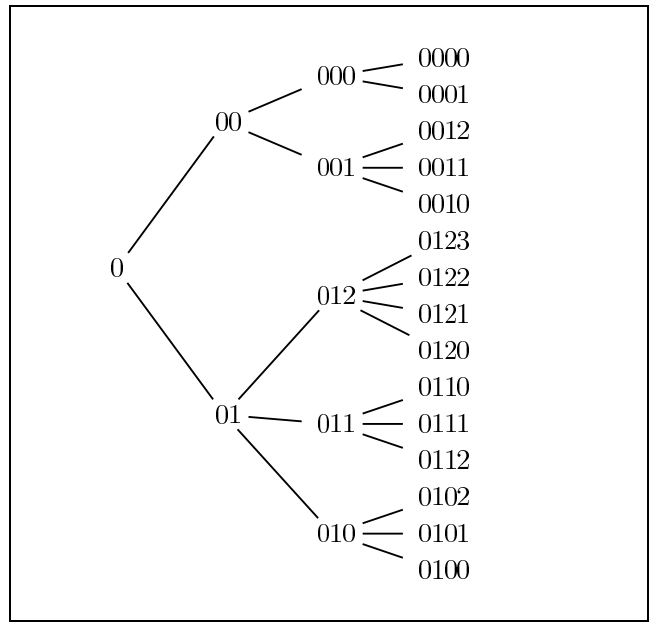}
\end{center}
\caption{The tree induced by the initial call {\tt Gen1}$(1,0,0,0)$ for $n=4$ and
generating the list $\mathcal A_4$.}
\label{fig:tree_pref}
\end{figure}

\clearpage
\section{The Co-Reflected Gray Code order for the sets $SE_n$, $A_n$, $R_n$, and $S_n$}
\label{sec:coRGC}
In this section we will consider, as in the previous one, the sets
$SE_n$, $A_n$, $R_n$ and $S_n$, but listed in ${\prec}_c$ order.
Our main goal is to prove that the obtained lists are Gray codes as well, 
and to develop generating algorithms for these lists.
%
Recall that $X_n$ generically denotes one of the sets $SE_n$, $A_n$, $R_n$, or $S_n$;
and  let $\Xtilde_n$ denote their corresponding list in $\prec_c$ order, that are, $\SEtilde_n$, $\Atilde_n$, $\Rtilde_n$, 
or $\Stilde_n$.
Clearly, a set of sequences listed in $\prec_c$ order is a \textit{suffix partitioned} list, 
that is, all sequences with same suffix are contiguous, and such are
the lists we consider here.

For a set $X_n$ and a sequence 
${\bsb b}=b_{k+1}b_{k+2}\ldots b_n$, we call ${\bsb b}$ an 
\textit{admissible suffix} in $X_n$ if there exists at least a sequence in $X_n$ having
suffix ${\bsb b}$. For example,   $124$ is an admissible suffix in $A_6$, because there 
are sequences in $A_6$ ending with $124$, namely $012124$ and $010124$. On the other 
hand, $224$ is not an admissible suffix in $A_6$; indeed, there is no length 6 ascent 
sequence ending with $224$.

We denote by $\mathcal \Xtilde_n\,|\,{\bsb b}$  the sublist of $\Xtilde_n$ of all 
sequences having  suffix ${\bsb b}$, and clearly, 
$\mathcal \Xtilde_n\,|\,{\bsb b}$ is a contiguous sublist of $\Xtilde_n$.
The set of all $x$ such that $x{\bsb b}$ is 
also an admissible suffix in $X_n$ is called the {\it defining set of the suffix} 
${\bsb b}$.

For $\prec$ order discussed in Section \ref{sec:RGC}, the characterization 
of prefixes is straightforward: $a_1a_2\ldots a_k$ is the prefix of some sequences 
in $X_n$, $n>k$, if and only if $a_1a_2\ldots a_k$ is in $X_k$. And the defining set of the prefix $a_1a_2\ldots a_k$ is $\{0,1,\ldots, \st(a_1a_2\ldots a_k)+1\}$. 
In the case of $\prec_c$ order, it turns out that similar notions are more complicated: 
for example, $13$ is an admissible suffix in $A_5$, but $13$ is not in $A_2$; 
and the defining set of the suffix $13$ is $\{0,2\}$, because $013$ and  $213$ are both admissible suffixes in $A_5$,
but $113$ is not. 
See Table \ref{tab:example_Atilde5} for the set $A_5$
listed in ${\prec}_c$ order.

\subsection{Suffix expansion of sequences in the sets $SE_n$, $A_n$, $R_n$, and $S_n$ }

For a suffix partitioned list, we need to build $\st$-restricted growth sequences under 
consideration from right to left, i.e., by expanding their suffix. For this purpose, we 
need the notions defined below.

\begin{De}
Let
${\bsb b}=b_{k+1}b_{k+2}\ldots b_n$, $1\leq k<n$, an admissible suffix in $X_n$.
\begin{itemize}
\item  $\alpha_{X}(\bsb{b})$ is the set of all  elements in the  
defining set of the suffix $\bsb b$. Formally:
$$
\alpha_{X}(\bsb{b})=\{x\,|\,x\bsb{b}\ {\rm is\ an\ admissible\ suffix\ in\ } X_n \},
$$
and for the empty suffix $\epsilon$,
$\alpha_{X}({\epsilon})=\{0,1,\ldots,n-1\}$.

\item {$\mu_{X}(\bsb{b})$ is the minimum required value of the statistic 
defining the set $X_n$, and provided by a length $(k+1)$ 
prefix of a sequence in $X_n$ having suffix ${\bsb b}$. Formally:
$$
\mu_{X}(\bsb{b})=\min\{\st(s_1s_2\ldots s_kb_{k+1})\,|\,s_1s_2\ldots s_k\bsb{b}\in X_n \}.   
$$
}
\end{itemize}
\end{De}

\noindent Notice that $\mu_{X}(x\bsb{b})\in \{ \mu_{X}(\bsb{b})-1,\mu_{X}(\bsb{b}),x\}$ for $x\in\alpha_X(\bsb b)$.

\begin{Rem}
\label{rem:coref}
Let $\st$ be one of statistics $\asc$, $\m$ or $\lv$, 
and $\bsb{s}=s_1s_2\ldots s_n$ be an $\st$-restricted growth sequence.
If there is a $k <n$ such that $s_{k+1}=k$, then 
$s_i=i-1$ for all $i$, $1\leq i\leq k$.
\end{Rem}
\proof
If $s_k<k-1$, then in each case for $\st$,
$\st(s_1s_2\ldots s_k)<k-1$, 
which is in contradiction with $s_{k+1}=k$, and so $s_k=k-1$.
Similarly, $s_{k-1}=k-2$, \ldots, $s_2=1$, and $s_1=0$.
\endproof

Under the conditions in the previous remark, 
$s_{k+1}=k$ imposes that all values at the left of
$k+1$ in $\bsb{s}$ are uniquely determined. 
As we will see  later, in the induced tree of the generating algorithm, all descendants of a node with $s_{k+1}=k$ have degree one, and we will eliminate the obtained degree-one path in order not to alter the algorithm efficiency.

It is routine to check the following propositions. 
(Actually, Proposition \ref{before_x_mu} is 
a consequence of Remark \ref{prop}.)

\begin{Pro}
\label{before_x_mu}
Let $X_n$ be one of the sets $SE_n$, $A_n$, $R_n$, or $S_n$. If ${\bsb b}=b_{k+1}b_{k+2}\ldots b_n$, $1\leq k<n$, is an admissible suffix in $X_n$, then $b_{k+1}\leq \mu_X(\bsb{b})$.
\end{Pro}



\begin{Pro}
\label{x_mu}
Let $Y_n$ be one of the sets $A_n$, $R_n$, or $S_n$.
If ${\bsb b}=b_{k+1}b_{k+2}\ldots b_n$, $1\leq k<n$, is an admissible suffix in $Y_n$, then

\begin{itemize}
\item[1] if $\bsb b=b_n$, that is, a length one admissible suffix, then $\mu_Y(\bsb b)=b_n$;
\item[2] $\mu_Y(\bsb{b})=k$ if and only if $b_{k+1}=k$;
\item[3] if $x\bsb{b}$ is also an admissible suffix in $Y_n$
      (i.e., $x\in \alpha_Y(\bsb{b})$) and $x\geq b_{k+1}$, then 
	$$\mu_Y(x\bsb{b})=\max\{x,\mu_Y(\bsb{b})\}.$$
\end{itemize}

\end{Pro}

The following propositions give
the values for $\alpha_X(\bsb{b})$ and $\mu_{X}(x\bsb{b})$, if $X_n$ is one of the sets $SE_n$, $A_n$, $R_n$, or $S_n$.
We do not  provide the proofs for Propositions \ref{pro:defseqSE_suff}, 
\ref{pro:muSE_parameter}, \ref{pro:defseq_suff_stair}, and 
\ref{pro:mu_parameter_stair}, because they are obviously 
based on the definition of the corresponding sequences.

\begin{Pro}
\label{pro:defseqSE_suff}
Let $\bsb{b}=\epsilon$ or $\bsb{b}=b_{k+1}b_{k+2}\ldots b_n$ be an admissible suffix in $SE_n$.
Then 
\begin{equation}
\nonumber
\label{eq:alpha_SEtilde}
\alpha_{SE}(\bsb{b})=\left\{ \begin {array}{ll}
\{0,1,,\ldots,n-1\} & {\rm if\ }  \bsb{b}=\epsilon,\\
\{0,1,\ldots,k-1\} & {\rm otherwise.}  
\end {array}
\right.
\end{equation}
\end{Pro}

\begin{Pro}
\label{pro:muSE_parameter}
Let $\bsb{b}=b_{k+1}b_{k+2}\ldots b_n$ be an admissible suffix in $SE_n$
and $x\in\alpha_{SE}(\bsb{b})$. Then

\begin{equation}
\nonumber
\label{eq:mu_SEtilde}
\mu_{SE}(x\bsb{b})=\mu_{SE}(\bsb b)-1. 
\end{equation}
\end{Pro}

\noindent Obviously, for a length one suffix $\bsb b=b_n$, it follows that $\mu_{SE}(\bsb b)=n-1$.

\begin{Exam}
If $\bsb b= \epsilon$, and $n=10$, then $\alpha_{SE}(\bsb b)=\{0,1,\ldots, 9\}$;\\
and for $\bsb b=b_{10}$, $b_{10}\in \{0,1,\ldots, 9\}$, it follows that $\mu_{SE}(x\bsb b)=9-1=8$, for all $x\in \alpha_{SE}(\bsb b)$.
\end{Exam}

\begin{Pro}
\label{pro:defseq_suff}
Let $\bsb{b}=\epsilon$ or $\bsb{b}=b_{k+1}b_{k+2}\ldots b_n$ be an admissible suffix in $A_n$.
Then 
\begin{equation}
\nonumber
\label{eq:alpha_Atilde}
\alpha_{A}(\bsb{b})=\left\{ \begin {array}{lll}
\{0,1,\ldots,n-1\} & {\rm if} & \bsb{b}=\epsilon,\\
\{k-1\}              & {\rm if} & \mu_{A}(\bsb{b})=k, {\rm \ or\ } \mu_A(\bsb b)=k-1 {\rm \ and\ } b_{k+1}=0,\\
\{0,1,\ldots,b_{k+1}-1\}\cup \{k-1\}    & {\rm if} & \mu_{A}(\bsb{b})= k-1 
                             {\rm \,and\ } 0<b_{k+1}<k, \\
\{0,1,\ldots,k-1\}  & {\rm if} & \mu_{A}(\bsb{b})<k-1.
\end {array}
\right.
\end{equation}
\end{Pro}

\proof
If $\bsb{b}=\epsilon$, the result is obvious.\\
For $\bsb{b}\neq\epsilon$, let
$x\in \alpha_{A}(\bsb{b})$.\\
\noindent
If $\mu_{A}(\bsb{b})=k$, by Proposition \ref{x_mu} point 2,
$b_{k+1}=k$ and by Remark \ref{rem:coref}
we have $x=k-1$.\\
\noindent If $\mu_{A}(\bsb{b})=k-1$ and $b_{k+1}=0$, then $\asc(xb_{k+1})=0$, and so 
$\mu_{A}(x\bsb{b})=\mu_{A}(\bsb{b})=k-1$,
and again by Proposition \ref{x_mu} point 2 we have $x=k-1$.\\
If $\mu_{A}(\bsb{b})=k-1$ and $0< b_{k+1}<k$, then there are two 
possibilities for $\mu_{A}(x\bsb{b})$:
\begin{itemize}
\item $\mu_{A}(x\bsb{b})=\mu_{A}(\bsb{b})=k-1$, if $\asc(xb_{k+1})=0$,
and as above $x=k-1$;
\item $\mu_{A}(x\bsb{b})=\mu_{A}(\bsb{b})-1=k-2$, if $\asc(xb_{k+1})=1$.  
In this case $x\in\{0,1,\ldots,b_{k+1}-1\}$. 
\end{itemize}
If $\mu_{A}(\bsb{b})<k-1$ (and consequently $0\le b_{k+1}<k$), 
then  there are two possibilities for $\mu_{A}(x\bsb{b})$:
\begin{itemize}
\item $\mu_{A}(x\bsb{b})=\mu_{A}(\bsb{b})$, if  $\asc(xb_{k+1})=0$, 
and we have $x\in\{b_{k+1},b_{k+1}+1,\ldots, k-1\}$;
\item $\mu_{A}(x\bsb{b})=\mu_{A}(\bsb{b})-1$, if  $\asc(xb_{k+1})=1$,
and we have $x\in\{0,1,\ldots,b_{k+1}-1\}$.
\end{itemize}
\endproof

\begin{Pro}
\label{eq:mu_Atilde}
Let $\bsb{b}=b_{k+1}b_{k+2}\ldots b_n$ be an admissible suffix in $A_n$
and $x\in\alpha_A(\bsb{b})$. Then
\begin{equation}
\nonumber
\mu_{A}(x\bsb{b})=\left\{ \begin {array}{lll}
x               & {\rm if} & x\geq \mu_{A}(\bsb{b}),\\
\mu_A(\bsb{b})    & {\rm if} & b_{k+1}\leq x< \mu_{A}(\bsb{b}), \\
\mu_A(\bsb{b})-1  & {\rm if} & x<b_{k+1}.
\end {array}
\right.
\end{equation}
\end{Pro}

\begin{proof}
If $x\ge \mu_{A}(\bsb{b})$, by Proposition \ref{before_x_mu}
it follows that $x\ge b_{k+1}$, and by Proposition \ref{x_mu} point 3, 
that $\mu_{A}(x\bsb{b})=\max\{x,\mu_{A}(\bsb{b})\}=x$.\\
If $b_{k+1}\le x<\mu_{A}(\bsb{b})$, then, again by Proposition \ref{x_mu} point 3,
it follows that $\mu_{A}(x\bsb{b})=\max\{x,\mu_{A}(\bsb{b})\}=\mu_{A}(\bsb{b})$.\\
If $x<b_{k+1}$, then $\asc(xb_{k+1})=1$, so 
$\mu_{A}(x\bsb{b})=\mu_{A}(\bsb{b})-1$.
\end{proof}

\begin{Exam}
Let $k=5$, $n=9$, and $\bsb{b}=b_6b_7b_8b_9=2050$ be an admissible
suffix in $A_9$. Clearly,  $\mu_A(\bsb{b})$, the minimum number of
ascents in a prefix $s_1s_2\ldots s_5b_6$ such that 
$s_1s_2\ldots s_5\bsb{b}\in A_9$, is $4$.

\noindent In this case, denoting $s_5$ by $x$, we have
\begin{itemize}
\item the set $\alpha_A(\bsb{b})$ of all possible values for $x$ 
is $\{0,1,,\ldots,b_{k+1}-1\}\cup \{k-1\}=\{0,1\}\cup\{4\}$. 
\item $\mu_A(x\bsb b)=\mu_A(\bsb b)-1=4-1=3$, if $x\in \{0,1\}$; 
or $\mu_A(x\bsb b)=\mu_A(\bsb b)=4$, if $x=4$.
\end{itemize}
\end{Exam}

\begin{Pro}
\label{pro:defseq_suff_rgf}
Let  $\bsb{b}=\epsilon$ or $\bsb{b}=b_{k+1}b_{k+2}\ldots b_n$ be an admissible suffix in $R_n$.
Then 
\begin{equation}
\nonumber
\label{eq:alpha_Rtilde}
\alpha_{R}(\bsb{b})=\left\{ \begin {array}{lll}
\{0,1,\ldots,n-1\} & {\rm if} & \bsb{b}= \epsilon,\\
\{k-1\}		& {\rm if} & \mu_{R}(\bsb{b})=k,\ {\rm or\ }
                  \mu_{R}(\bsb{b})=k-1 \ {\rm and}\ b_{k+1}<k-1, \\
\{0,1,\ldots,k-1\}    & {\rm if} & \mu_{R}(\bsb{b})= k-1 \ {\rm and}\  b_{k+1}=k-1,\ {\rm or}\ 
                        \mu_{R}(\bsb{b})<k-1.
\end {array}
\right.
\end{equation}
\end{Pro}

\proof
If $\bsb{b}=\epsilon$, the result is obvious.\\
For $\bsb{b}\neq\epsilon$, let $x\in \alpha_{R}(\bsb{b})$.\\
If $\mu_{R}(\bsb{b})=k$, by Proposition \ref{x_mu} point 2,
$b_{k+1}=k$ and by Remark \ref{rem:coref}
we have $x=k-1$.\\
If $\mu_{R}(\bsb{b})=k-1$ and $b_{k+1}<k-1$, then 
the maximal value of the statistic $\m$ (defining the set $R_n$)
of a length $k+1$ prefix ending with $b_{k+1}<k-1$ is
$k-1$, and it is reached when $x=k-1$.\\
If $\mu_{R}(\bsb{b})=k-1 \,{\rm and}\, b_{k+1}=k-1$, then there are two 
possibilities for $\mu_{R}(x\bsb{b})$:
\begin{itemize}
\item $\mu_{R}(x\bsb{b})=\mu_{R}(\bsb{b})=k-1$, and as above, 
   this implies $x=k-1$;
\item $\mu_{R}(x\bsb{b})=\mu_{R}(\bsb{b})-1=k-2$, which implies 
   $x\in\{0,1,\ldots,k-2\}$.
\end{itemize}
Finally, if $\mu_{R}(\bsb{b})<k-1$, then $x$ can be any value in $\{0,1,\ldots,k-1\}$.
\endproof

\begin{Pro}
\label{pro:mu_parameter_rgf}
Let $\bsb{b}=b_{k+1}b_{k+2}\ldots b_n$ be an admissible suffix in $R_n$
and $x\in\alpha_{R}(\bsb{b})$. Then

\begin{equation}
\nonumber
\label{eq:mu_Rtilde}
\mu_{R}(x\bsb{b})=\left\{ \begin {array}{lll}

x               & {\rm if} & x\geq \mu_{R}(\bsb{b}), \\
\mu_{R}(\bsb{b})    & {\rm if} & b_{k+1}\leq x< \mu_{R}(\bsb{b})\ {\rm or\ }
                               x<b_{k+1} <\mu_{R}(\bsb{b}),\\
\mu_{R}(\bsb{b})-1  & {\rm if} & x<b_{k+1}=\mu_{R}(\bsb{b}).\\
\end {array}
\right.
\end{equation}
\end{Pro}

\proof
The case $x\ge\mu_{R}(\bsb{b})$ is analogous with the similar case in Proposition 
\ref{eq:mu_Atilde}.\\
The next case is equivalent with $x< \mu_{R}(\bsb{b})$ and $b_{k+1}<\mu_{R}(\bsb{b})$,
and since $R_n$ corresponds to the statistic $\m$, the result holds.\\
Finally, if $x<b_{k+1} =\mu_{R}(\bsb{b})$, then 
$\mu_{R}(\bsb{b})=\mu_{R}(x\bsb{b})+1$, and so 
$\mu_{R}(x\bsb{b})=\mu_{R}(\bsb{b})-1$. 
\endproof

\begin{Exam}
Let $k=4$, $n=7$, and $\bsb{b}=b_5b_6b_7=241$ be an admissible
suffix in $R_7$. It follows that $b_{k+1}=2$, $\mu_R(\bsb{b})=3$ and
\begin{itemize} 
\item $\alpha_R(\bsb{b})=\{k-1\}=\{3\}$;
\item $\mu_R(x\bsb b)=x=3$.
\end{itemize}
\end{Exam}

\begin{Pro}
\label{pro:defseq_suff_stair}
Let $\bsb b=\epsilon$ or $\bsb{b}=b_{k+1}b_{k+2}\ldots b_n$ be an admissible suffix in $S_n$.
Then 
\begin{equation}
\nonumber
\alpha_{S}(\bsb{b})=\left\{ \begin {array}{lll}
\{0,1,\ldots,n-1\} & {\rm if} & \bsb{b}=\epsilon,\\
\{k-1\} & {\rm if} & b_{k+1}=k,\\
\{C,C+1,\ldots,k-1\}		& {\rm if} & 0\le b_{k+1}\le k-1,
\end {array}
\right.
\end{equation}
where $C=\max\{0,b_{k+1}-1\}$.
\end{Pro}

\noindent Since $\mu_{S}(\bsb{b})=b_{k+1}$, the next result follows:

\begin{Pro}
\label{pro:mu_parameter_stair}
Let $\bsb{b}=b_{k+1}b_{k+2}\ldots b_n$ be an admissible suffix in $S_n$
and $x\in\alpha_{S}(\bsb{b})$. Then
\begin{equation}
\nonumber
\mu_{S}(x\bsb{b})=x.
\end{equation}
\end{Pro}

\begin{Exam}
Let $k=6$, $n=9$, and $\bsb{b}=b_7b_8b_9=457$ be an admissible
suffix in $S_9$. So we have $\mu_S(\bsb{b})=4$, $C=\max\{0,b_{k+1}-1\}=\max\{0,3\}=3$, and
\begin{itemize}
\item $\alpha_S(\bsb{b})=\{C,C+1,\ldots,k-1\}=\{3,4,5\}$;
\item $\mu_S(x\bsb b)=x$, where $x\in \{3,4,5\}$.
\end{itemize}
\end{Exam}

\subsection{The bound of Hamming distance between successive sequences in the lists $\SEtilde_n$, 
$\Atilde_n$, $\Rtilde_n$, and $\Stilde_n$}

Now we show that the Hamming distance between two successive
sequences in the mentioned lists is upper bounded by a constant,
which implies that the lists are Gray codes.
These results are embodied in Theorems \ref{The:SEtilde}, \ref{The:ARtilde} and \ref{The:Stilde}.

%

\subsubsection{The list $\SEtilde_n$}
\begin{The}
\label{The:SEtilde}
The list $\SEtilde_n$ is $1$-Gray code.
\end{The}

\proof
The result follows from the fact that 
the restriction of the 1-Gray code list $\mathcal G_n(n)$ to any product space remains a 1-Gray code (see \cite{VajnovVernay}), in particular to the set
$$V_n=\vartheta_1 \times \vartheta_2 \times \ldots  \times \vartheta_n,$$ where
\begin{itemize}
\item $\vartheta_i=\{0,1,...,n-i\}$, if $n$ is odd and $i$ is even, or
\item $\vartheta_i=\{i-1,i, ...,n-1\}$, if $n$ is even, or $n$ and $i$ are both odd.
\end{itemize}
Then by applying to each sequence $\bsb s$ in the list $\mathcal V_n$ the two transforms mentioned before Definition \ref{De:CoRGC}, namely:
\begin{itemize}
\item complementing each digit in $\bsb s$ if $n$ is even, or only digits in odd positions
if $n$ is odd, then
\item reversing the obtained sequence,
\end{itemize}
the desired $1$-Gray code for the set $SE_n$ in $\prec_c$ order is obtained.
\endproof

\subsubsection{The lists $\Atilde_n$ and $\Rtilde_n$}
The next proposition describes the pattern of 
$\bsb s=\last(\Ytilde_n\,|\,{\bsb b})$ and $\bsb s=\first(\Ytilde_n\,|\,{\bsb b})$, 
where $\Ytilde_n$ is one of the lists $\Atilde_n$ or $\Rtilde_n$.

\begin{Pro}
\label{pro:CorefLast}
Let $Y_n$ be one of the sets $A_n$ or $R_n$, and 
${\bsb b}=b_{k+1}b_{k+2}\ldots b_n$ be an admissible suffix in $Y_n$.
If $\bsb s=\last(\Ytilde_n\,|\,{\bsb b})$ or $\bsb s=\first(\Ytilde_n\,|\,{\bsb b})$, 
then $\bsb s$ has one of the following patterns:
\begin{itemize}
  \item $\bsb s=012\ldots (k-2)(k-1){\bsb b}$, or
  \item $\bsb s=012\ldots(k-2)0{\bsb b}$.		
\end{itemize}	
\end{Pro}

\proof
Let ${\bsb s}=s_1s_2\ldots s_kb_{k+1}b_{k+2}\ldots b_n$.
Since $\bsb{s}=\last(\Ytilde_n\,|\,{\bsb b})$ or $s=\first(\Ytilde_n\,|\,{\bsb b})$,
according to $\alpha_Y(\bsb b)$ given in Propositions 
\ref{pro:defseq_suff} and \ref{pro:defseq_suff_rgf},
it follows that $s_k\in\{0,k-1\}$. In other words, $s_k$ is either the smallest or the largest value in $\alpha_Y(\bsb b)$.\\
\noindent If $s_k=k-1$, then by Remark \ref{rem:coref} we have 
$\bsb s=012\ldots (k-2)(k-1){\bsb b}$.

\noindent 
If $s_k=0$, then considering the definition of $\prec_c$ order we have either
\begin{itemize}
\item $\bsb{s}=\first(\Ytilde_n\,|\,{\bsb b})$ and 
      $\sum_{i=k+1}^{n}b_i+(n-k)$ is odd, or
\item $\bsb{s}=\last(\Ytilde_n\,|\,{\bsb b})$ and 
      $\sum_{i=k+1}^{n}b_i+(n-k)$ is even.
\end{itemize}
For the first case, again by the definition of $\prec_c$ order,
it follows that $s_{k-1}$ must be the largest value in 
$\alpha_Y(0\bsb{b})$, and so $s_{k-1}=k-2$, and 
by Remark \ref{rem:coref}, $\bsb s=012\ldots (k-2)0{\bsb b}$.
Similarly, the same result is obtained for the second case.
\endproof

A direct consequence of the previous proposition is 
the next theorem.

\begin{The}
\label{The:ARtilde}
The lists $\Atilde_n$ and $\Rtilde_n$ are $2$-adjacent Gray codes.
\end{The}
\proof
Let 
$\bsb s,\bsb t \in Y_n$, with $\bsb t=\tildesuc_Y(\bsb s)$.
If $k+1$ is the rightmost position where $\bsb s$ and $\bsb t$ differ,
then there are admissible suffixes  
${\bsb b}=b_{k+1}b_{k+2}\ldots b_n$ and ${\bsb b'}=b'_{k+1}b_{k+2}\ldots b_n$
in $Y_n$ such that $\bsb s=\last(\Ytilde_n|{\bsb b})$ and 
$\bsb t=\first(\Ytilde_n|{\bsb b'})$.

\noindent By Proposition \ref{pro:CorefLast}, $\bsb s$ has pattern
\begin{equation}
\nonumber
\begin{array}{l}
		012\ldots (k-2)(k-1){\bsb b}, {\rm \,\,or\ }\\
		012\ldots(k-2)0{\bsb b};
\end{array}
\end{equation}
and $\bsb t$ has pattern
\begin{equation}
\nonumber
\begin{array}{l}
		012\ldots (k-2)(k-1){\bsb b'}, {\rm \,\,or\ }\\
		012\ldots(k-2)0{\bsb b'}.
\end{array}
\end{equation}
And in any case, $\bsb s$ and $\bsb t$ differ in position 
$k+1$ and possibly in position $k$.

%
%
\endproof

\subsubsection{The list $\Stilde_n$}
The next proposition gives the pattern of 
$\last(\Stilde_n\,|\,{\bsb b})$ and 
$\first(\Stilde_n\,|\,{\bsb b})$ for an admissible suffix 
$\bsb b$ in $S_n$.

\begin{Pro}
\label{pro:CorelSnLast}
Let ${\bsb b}=b_{k+1}b_{k+2}\ldots b_n$ be an admissible suffix in $S_n$.
If $\bsb s=\last(\Stilde_n\,|\,{\bsb b})$, 
then the pattern of $\bsb s$ is given by:

\begin{itemize}
\item if $b_{k+1}=k$\ or\ $\sum_{i=k+1}^{n}b_i+(n-k)$ is odd, then
\begin{equation} \nonumber
\bsb s=012\ldots (k-2)(k-1){\bsb b};
\end{equation}
\item if $b_{k+1}< k$ and $\sum_{i=k+1}^{n}b_i+(n-k)$ is even, and
either $b_{k+1}=0$ or $b_{k+1}$ is odd, then
\begin{equation} \nonumber
\bsb s=012\ldots (k-2)(\max\{0,b_{k+1}-1\}){\bsb b};
\end{equation}
\item if $b_{k+1}< k$ and $\sum_{i=k+1}^{n}b_i+(n-k)$ is even, and
      $b_{k+1}>0$ is even, then
\begin{equation} \nonumber
\bsb s=012\ldots (k-3)(b_{k+1}-2)(b_{k+1}-1){\bsb b}.
\end{equation}
\end{itemize}
Similar results hold for $\bsb s=\first(\Stilde_n\,|\,{\bsb b})$ by replacing `odd' by `even', and {\em vice versa}, for the parity of $\sum_{i=k+1}^{n}b_i+(n-k)$.
\end{Pro}

\proof
\noindent
Let $\bsb s=s_1s_2\ldots s_kb_{k+1}\ldots b_n=\last(\Stilde_n\,|\,{\bsb b})$. \\
If $b_{k+1}=k$ or $\sum_{i=k+1}^{n}b_i+(n-k)$ is odd, then 
$s_k$ is the largest value in $\alpha_S(\bsb b)$, so $s_k=k-1$,
and by Remark \ref{rem:coref}, 
	      $s_i=i-1$ for $1\le i\le k$. So the first case holds.\\    
If $\sum_{i=k+1}^{n}b_i+(n-k)$ is even and $b_{k+1}< k$, then 
$s_k$ is the smallest value in $\alpha_S(\bsb b)$, 
namely $\max\{0,b_{k+1}-1\}$, which is even if $b_{k+1}=0$ or $b_{k+1}$ is odd.
Thus, by the definition of $\prec_c$ order, $s_{k-1}$
is the largest value in $\alpha_S(s_k\bsb b)$, which is $k-2$,
and by Remark \ref{rem:coref}, the second case holds.\\
For the last case, as above, $s_k=\max\{0,b_{k+1}-1\}$, and considering $b_{k+1}>0$ and even, it follows that 
$s_k=b_{k+1}-1$ is odd. Thus $s_{k-1}$ is the minimal value in 
$\alpha_S(s_k\bsb b)$, that is $b_{k+1}-2$, which in turns is even,
and the last case holds.\\
The proof for the case  $\bsb s=\first(\Stilde_n\,|\,{\bsb b})$ is similar.
\endproof

\begin{The}
\label{The:Stilde}
The list $\Stilde_n$ is $3$-adjacent Gray codes.
\end{The}
\proof
Let ${\bsb t}=\tildesuc_S(\bsb s)$, and
$k+1$ be the rightmost position where ${\bsb s}$ and ${\bsb t}$
differ. Let us denote by $\bsb b$ the length $(n-k)$ suffix
of ${\bsb s}$ and $\bsb b'$ that of ${\bsb t}$;
so, $\bsb s=\last({\bsb b}\,|\,\mathcal S_n)$ and
$\bsb t=\first({\bsb b'}\,|\,\mathcal S_n)$. It follows by Proposition \ref{pro:CorelSnLast}, that
$s_i=t_i=i-1$ for all $i\leq k-2$, and so the other differences possibly occur in position $k$ and in position $k-1$. 

Considering all valid combinations for
   $\bsb s$ and $\bsb t$ as given in Proposition \ref{pro:CorelSnLast},
   the proof of the adjacency is routine, and based
   on the following: $s_k\neq t_k$ if and only if
   $s_{k-1}\neq t_{k-1}$.
   It follows that $\bsb s$ and $\bsb t$ differ in one position,
   or three positions which are adjacent.

\endproof

In addition, the lists $\Atilde_n$, $\Rtilde_n$, and $\Stilde_n$, are circular Gray codes. This is a consequence of the following remarks based on Propositions \ref{pro:CorefLast} and \ref{pro:CorelSnLast}:
\begin{itemize}
\item $\first(\Ytilde_n)=012\ldots (n-2)(n-1)$;
\item $\last(\Ytilde_n)=012\ldots (n-2)0$;
\end{itemize}
where $\Ytilde_n$ is one of the lists $\Atilde_n$, $\Rtilde_n$, or $\Stilde_n$.

\subsection{Generating algorithm for $\SEtilde_n$, $\Atilde_n$, 
$\Rtilde_n$, and $\Stilde_n$}

Here we explain algorithm {\tt Gen2} in Figure \ref{fig:Gen2} 
which generates suffix partitioned Gray codes for restricted growth sequences;
according to particular instances of the functions called by 
it, {\tt Gen2} produces the list $\SEtilde_n$, $\Atilde_n$, 
$\Rtilde_n$, or $\Stilde_n$.
Actually, for convenience, {\tt Gen2} produces length $(n+1)$ sequences
${\bsb s}=s_1s_2\ldots s_{n+1}$ with $s_{n+1}=0$, and so, neglecting the last value in each 
sequence $\bsb s$ the desired list is obtained.
Notice that with this dummy value for $s_{n+1}$ we have
$\mu_X(s_ks_{k+1}\ldots s_n)=\mu_X(s_ks_{k+1}\ldots s_n0)$,
for $k\leq n$, and similarly for $\alpha_X$.

In {\tt Gen2}, the sequence $\bsb s$ is a global variable, and initialized by 
$01\ldots (n-1)0$, which is the first length $n$ sequence in $\prec_c$ order, followed by a $0$;
and  the main call is {\tt Gen2}$(n+1,0,0,0)$. 
Procedure {\tt Gen2} has the following parameters
(the first three of them are similar with those of procedure 
{\tt Gen1}):
\begin{itemize}
\item $k$, the position in the sequence $\bsb s$ which is updated by the current call;
\item $x$, the value to be assigned to $s_k$;
\item $dir$, gives the direction in which $s_{k-1}$ covers 
      $\alpha_X(s_ks_{k+1}\ldots s_n0)$,
      the defining set of the current suffix;      
\item $v$, the value of $\mu_X(s_{k+1}\ldots s_n0)$.
\end{itemize}

\noindent
The functions called by {\tt Gen2} are given in Figures \ref{fig:f-ARtilde} 
and \ref{fig:f-S-SEtilde}. They are principally 
based on the evaluation of $\alpha_X$ and $\mu_X$ for 
the current suffix of $\bsb s$, and are:
\begin{itemize}
\item ${\tt Mu\_}X(k,x,v)$ returns the value of $\mu_X(s_ks_{k+1}\ldots s_n0)$, with
      $x=s_k$. 
\item ${\tt IsDegreeOne\_}X(k)$ stops the recursive calls
when $\alpha_X(s_ks_{k+1}\ldots s_n0)$ has only one
element, namely $k-2$.
In this case, by Remark \ref{rem:coref}
the sequence is
uniquely determined by the current suffix, and
this prevents {\tt Gen2} to produce degree one calls.
In addition, ${\tt IsDegreeOne\_}X(k)$ sets appropriatelly $d-1$ values at
the left of $s_k$, where $d$ is the upper bound of the Hamming distance in the
list (changes at the left of $s_1$ are considered with no effect).
This can be considered as a {\it Path Elimination Technique} \cite{Rus}.
\item ${\tt Lowest\_}X(k)$ is called when {\tt IsDegreeOne\_}$X$ returns {\tt false},
      and gives the lowest value in $\alpha_X(s_ks_{k+1}\ldots s_n0)$.
\item ${\tt SecLargest\_}X(k,u)$, is called when {\tt IsDegreeOne\_}$X$ returns {\tt false}, 
      and gives the second largest value in 
      $\alpha_X(s_ks_{k+1}\ldots s_n0)$ (the largest value being 
      always $k-2$).

\end{itemize}
 

\noindent
By this construction, algorithm {\tt Gen2} has no degree one calls and it
satisfies the CAT principle.
Figure \ref{fig:tree_suff} shows the tree induced by the algorithm when
generates $\Atilde_4$.

\begin{figure}[h]
\centering
\fbox{
\begin{minipage}{0.8\linewidth} 
\tt
\begin{tabbing}
procedure Gen2$(k,x,dir,v)$\\
global $n$,${\bsb s}$;\\
$s_k:=x$;\\
$u:={\tt Mu\_}X(k,x,v)$;\\
if IsDegreeOne\_$X(k,u)$\\
then Print ${\bsb s}$;\\
else \=$c$:=Lowest\_$X(k)$;\\
	       \>$d$:=SecLargest\_$X(k,u)$;\\
	       \>if $dir\mod 2=1$\\
	       \>then \=for \=$i:=c$ to $d$ do\\
	       \>	       \>	     \>Gen2$(k-1,i,i,u)$;\\
	       \>	       \>enddo\\
	       \>endif\\
	       \>Gen2$(k-1,k-2,(k-1-dir)\mod 2,u)$;\\
	       \>if $dir\mod 2=0$\\
	       \>then \=for \=$i:=d$ downto $c$ do\\
	       \>	       \>	     \>Gen2$(k-1,i,i+1,u)$;\\
	       \>	       \>enddo\\
	       \>endif\\
endif\\
end.
\end{tabbing}
\end{minipage}
}
\caption{Algorithm {\tt Gen2}, generating the list $\Xtilde_n$.}
\label{fig:Gen2}
\end{figure}


\begin{figure}[t]
\centering
\begin{tabular}{cc}

	\fbox{
	\begin{minipage}[t][13cm]{7cm}
	\tt
	\begin{tabbing}
	fu\=nction Mu\_A$(m,i,w)$\\
\>		if $i\ge w$ \\
\>		then return $i$;\\
\>		else \=if $i\ge s_{m+1}$\\
\>			\>then return $w$;\\
\>		        \>else return $w-1$;\\
\>			\> endif\\
\>		endif\\
		end.\\
		function IsDegreeOne\_A$(m,v)$\\
\>		if \=$v=m-1$ or\\
\>			\>$(v=m-2$ and $s_m=0)$\\
\>		then \=$s_{m-1}:=m-2$;\\
\>			\>return {\rm true};\\
\>		else return {\rm false};\\
\>		endif\\
	end.\\

	function Lowest\_A$(m)$\\	
\>	return 0;\\
	end.\\

	function SecLargest\_A$(m,w)$\\
\>	if \=$w=m-2$ and $s_m>0$\\
\>   \> and $s_m<m-1$\\
\>	then return $s_m-1$;\\
\>	else return $m-3$;\\
\>	endif\\
	end.\\
	\phantom{text text text text text text}
	\end{tabbing}
	\end{minipage}
	}
 & 
\fbox{
	\begin{minipage}[t][13cm]{7cm}
	\tt
	\begin{tabbing}
	fu\=nction Mu\_R$(m,i,w)$\\
\>		if $i\ge w$ \\
\>		then return $i$;\\
\>		else \=if $s_{m+1}<w$\\
\>				\>then return $w$;\\
\>				\>else return $w-1$;\\
\>				\>endif\\
\>		endif\\
	end.\\

	function IsDegreeOne\_R$(m,v)$\\
\>	if \=$v=m-1$ or\\
\>		\>$(v=m-2$ {\tt and} $s_{m}<m-2)$\\
\>	then \=$s_{m-1}:=m-2$;\\
\>			\>return {\rm true};\\
\>	else return {\rm false};\\
\>	endif\\
	end.\\

	function Lowest\_R$(m)$\\	
\>	return 0;\\
	end.\\

	function SecLargest\_R$(m,w)$\\
\>	 return $m-3$;\\
	end.\\
	\phantom{text text text text text text}
	\end{tabbing}	
	\end{minipage}
	}
\\
(a) & (b)
\end{tabular}
\caption{Particular functions called by {\tt Gen2}, generating the lists: (a) $\Atilde_n$, and (b) 
$\Rtilde_n$.}
	\label{fig:f-ARtilde}
\end{figure}

\begin{figure}[t]
\centering
\begin{tabular}{cc}
	\fbox{
	\begin{minipage}[t][9.5cm]{7cm}
	\tt
	\begin{tabbing}
	fu\=nction Mu\_S$(m,i,w)$\\
\>		return $i$;\\
		end.\\

	function IsDegreeOne\_S$(m,v)$\\
\>	if $s_m=m-1$\\
\>	then \=$s_{m-1}:=m-2$;\\
\>			\>$s_{m-2}:=m-3$;\\
\>			\>return {\rm true};\\
\>	else return {\rm false};\\
\>	endif\\
	end.\\

	function Lowest\_S$(m)$\\
\>	if \=$s_{m}>1$ and $s_m\le m-2$\\

\>	then return $s_m-1$;\\
\>	else return 0;\\
\>	endif\\
	end.\\

	function SecLargest\_S$(m,w)$\\
\>	return $m-3$;\\
	end.\\
	\phantom{text text text text text te}
	\end{tabbing}
	\end{minipage}
	}
	\label{fig:f-Stilde}
&
	\fbox{
	\begin{minipage}[t][9.5cm]{7cm}
	\tt
	\begin{tabbing}
		fu\=nction Mu\_{SE}$(m,i,w)$\\
\>		if $m=n$\\
\>		then return $n-1$;\\
\>		else return $w-1$;\\
		end.\\

	function IsDegreeOne\_{SE}$(m,v)$\\
\>	if $m=2$\\
\>	then return {\rm true};\\
\>	else return {\rm false};\\
	end.\\

	function Lowest\_SE$(m)$\\	
\>	return 0;\\
	end.\\

	function SecLargest\_SE$(m,w)$\\
\>	return $m-3$;\\
	end.\\
	\phantom{text text text text text te}
	\end{tabbing}	
	\end{minipage}
	}
\\
(a) & (b)
\end{tabular}
\caption{Particular functions called by {\tt Gen2}, generating the lists: (a) $\Stilde_n$, and 
(b) $\SEtilde_n$.}	
\label{fig:f-S-SEtilde}
\end{figure}

\begin{figure}[h]
\begin{center}\includegraphics[width=8.5cm]{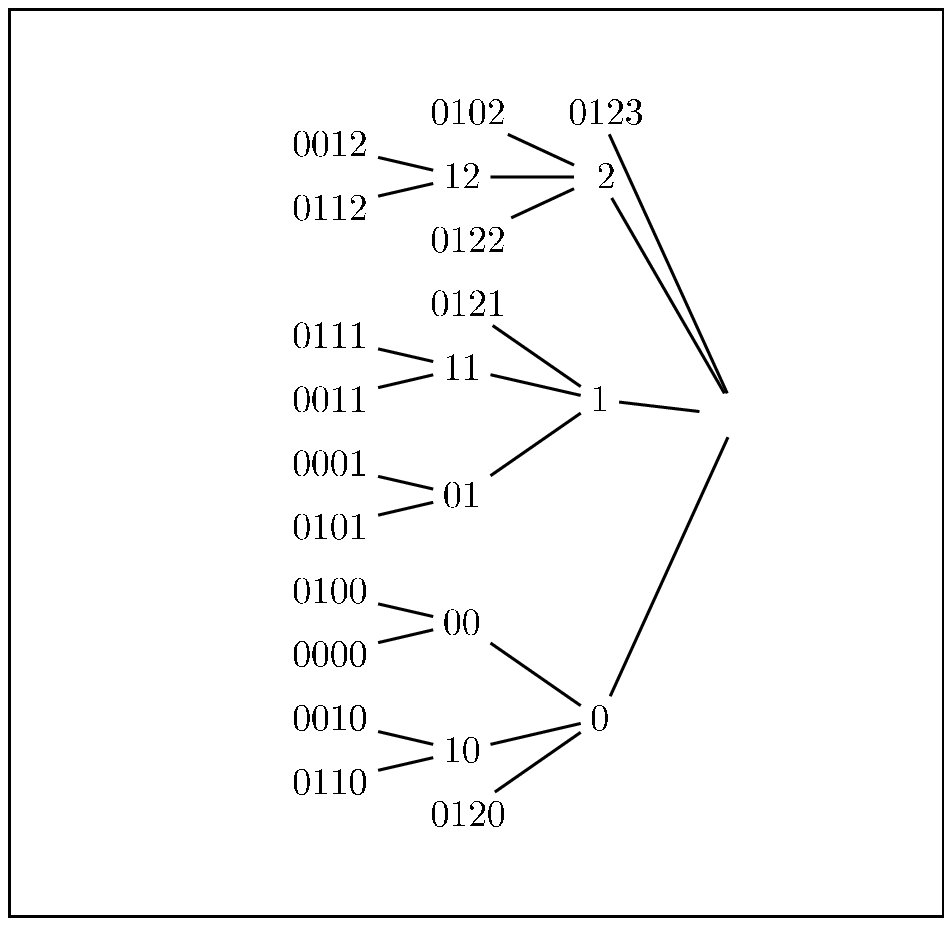}
\end{center}
\caption{The tree induced by the call {\tt Gen2}$(5,0,0,0)$, 
generating the list $\Atilde_4$.}
\label{fig:tree_suff}
\end{figure}

\clearpage
\section{Final remarks}

We conclude this paper by comparing for each of the sets 
$SE_ n$, $A_n$, $R_n$ and $S_n$ the prefix partitioned Gray codes 
induced by $\prec$ order and the suffix partitioned one induced
by $\prec_c$ order.
This can be done by comparing the Hamming distance 
between all pairs of successive sequences either in 
the worst case or in average.

Table \ref{Tb1} summarizes Theorems \ref{The_Gray_A_others},
\ref{The:ARtilde}, and \ref{The:Stilde}, and Proposition \ref{G_SE}, and 
gives the upper bound of the Hamming distance 
(that is, the worst case Hamming distance) for the two 
order relations. It shows that for the sets $SE_ n$ and $S_n$ these relations
have same performances, and for the sets $A_n$ and $R_n$, 
$\prec_c$ order induces more restrictive Gray codes.

\begin{table}[h]
\centering
\begin{tabular}{c|c|c}

\hline
Set	& $\prec$ (RGC)    & $\prec_c$ (Co-RGC)\\
 	&  order 	   &  order\\
\hline\hline
$SE_ n$ 	&	1	& 1\\
$A_n$		&	3	&2\\
$R_n$		&	3	&2\\
$S_n$		&	3	&3\\
\hline
\end{tabular}
\caption{\label{Tb1}
The bound of the Hamming distance between two successive 
sequences in $\prec$ and $\prec_c$ orders.}
\end{table}

For a list $\mathcal L$ of sequences, the {\it average Hamming distance} 
is defined as

$$
\frac{\sum d({\bsb s},{\bsb t})}{N-1},
$$
where the summation is taken over all $\bsb s$ in $\mathcal L$, except its 
last element, $\bsb t$ is the successor of $\bsb s$ in $\mathcal L$, $d$
is the Hamming distance, and $N$ the number of sequences in $\mathcal L$.

Surprisingly, despite $\prec_c$ order has same or
better performances in terms of worst case Hamming distance, 
if we consider the average Hamming distance,
numerical evidences show that $\prec$ order is `more optimal' than $\prec_c$ order on $A_n$, $R_n$ ($n\geq 5$),
and $S_n$ ($n\geq 6$). And this phenomenon strengthens for large $n$;
see Table \ref{Tb2}.

\begin{table}[h]
\centering
\begin{tabular}{|l|*{4}{c|}*{4}{c|}}
\hline
\multirow{3}{*}{$n$} & \multicolumn{4}{|c|}{ $\prec$ (RGC) order } &
\multicolumn{4}{|c|}{ $\prec_c$ (Co-RGC) order } \\

\cline{2-9}
  & \multirow{2}{*}{$\mathcal{SE}_n$} & \multirow{2}{*}{$\mathcal A_n$} & \multirow{2}{*}{$\mathcal R_n$} & \multirow{2}{*}{$\mathcal S_n$} & \multirow{2}{*}{$\SEtilde_n$} & \multirow{2}{*}{$\Atilde_n$} & \multirow{2}{*}{$\Rtilde_n$} & \multirow{2}{*}{$\Stilde_n$}\\
 & & & & & & & & \\
\hline\hline
4 & 1 &1.21  & 1.21  & 1.31  & 1 & 1.14  & 1.14 & 1.15 \\
5& 1 & 1.13 & 1.12 & 1.29 &  1 &  1.19 & 1.18 & 1.24  \\
6&1 & 1.09 & 1.07 & 1.27 &  1 &  1.23 & 1.20 & 1.31 \\
7&1 & 1.06 & 1.06 &  1.26 &  1 &  1.25 & 1.22 & 1.35\\
8&1 &  1.04& 1.04 &  1.25 &  1 &  1.26 & 1.23 &  1.37\\
9&1&  1.03& 1.03 &  1.24&  1 &  1.28 & 1.24 &  1.39\\
10& 1& 1.02 & 1.03  & 1.23  & 1 & 1.28  & 1.24 & 1.41 \\
\hline
\end{tabular}
\caption{\label{Tb2}The average Hamming distance for $\prec$ order and 
$\prec_c$ order.}
\end{table}

Algorithmically, $\prec_c$ order has the advantage that its corresponding generating
algorithm, {\tt Gen2}, is more appropriate to be parallelized
than its  $\prec$ order counterpart, {\tt Gen1}. 
Indeed, the main call of {\tt Gen2} produces $n$ recursive calls
(compare to two recursive calls produced by the main call of {\tt Gen1}),
and so we can have more parallelized computations; and this is
more suitable for large $n$. 
See Figure \ref{fig:tree_pref} and \ref{fig:tree_suff}
for examples of computational trees.

Finally, it will be of interest to 
explore order relation based Gray codes
for restricted growth sequences defined by statistics other than 
those considered in this paper.
In this vein we suggest the following
conjecture, checked by computer for $n\leq 10$, and  concerning descent sequences 
(defined similarly with ascent sequences in  Section 2).

\begin{Con}
The set of length $n$ descent sequences
listed in $\prec_c$ order is a $4$-adjacent Gray code.
\end{Con}



\clearpage
\section*{Appendix}

\begin{table}[!h]
\centering
\begin{tabular}{c|c|c|c}
\hline
Sequence    &   $S_5$       &$R_5$      &$A_5$\\
\hline\hline
00000   &   $\checkmark$&   $\checkmark$&$\checkmark$\\
00001   &   $\checkmark$&   $\checkmark$&$\checkmark$\\
00012   &   $\checkmark$&   $\checkmark$&$\checkmark$\\
00011   &   $\checkmark$&   $\checkmark$&$\checkmark$\\
00010   &   $\checkmark$&   $\checkmark$&$\checkmark$\\
00123   &   $\checkmark$&   $\checkmark$&$\checkmark$\\
00122   &   $\checkmark$&   $\checkmark$&$\checkmark$\\
00121   &   $\checkmark$&   $\checkmark$&$\checkmark$\\
00120   &   $\checkmark$&   $\checkmark$&$\checkmark$\\
00110   &   $\checkmark$&   $\checkmark$&$\checkmark$\\
00111   &   $\checkmark$&   $\checkmark$&$\checkmark$\\
00112   &   $\checkmark$&   $\checkmark$&$\checkmark$\\
00102   &           &   $\checkmark$&$\checkmark$\\
00101   &   $\checkmark$&   $\checkmark$&$\checkmark$\\
00100   &   $\checkmark$&   $\checkmark$&$\checkmark$\\
01230   &   $\checkmark$&   $\checkmark$&$\checkmark$\\
01231   &   $\checkmark$&   $\checkmark$&$\checkmark$\\
01232   &   $\checkmark$&   $\checkmark$&$\checkmark$\\
\hline
\end{tabular}
\quad
\begin{tabular}{c|c|c|c}
\hline
Sequence    &   $S_5$       &$R_5$      &$A_5$\\
\hline\hline
01233   &   $\checkmark$&   $\checkmark$&$\checkmark$\\
01234   &   $\checkmark$&   $\checkmark$&$\checkmark$\\
01223   &   $\checkmark$&   $\checkmark$&$\checkmark$\\
01222   &   $\checkmark$&   $\checkmark$&$\checkmark$\\
01221   &   $\checkmark$&   $\checkmark$&$\checkmark$\\
01220   &   $\checkmark$&   $\checkmark$&$\checkmark$\\
01210   &   $\checkmark$&   $\checkmark$&$\checkmark$\\
01211   &   $\checkmark$&   $\checkmark$&$\checkmark$\\
01212   &   $\checkmark$&   $\checkmark$&$\checkmark$\\
01213   &           &   $\checkmark$&$\checkmark$\\
01203   &           &   $\checkmark$&$\checkmark$\\
01202   &           &   $\checkmark$&$\checkmark$\\
01201   &   $\checkmark$&   $\checkmark$&$\checkmark$\\
01200   &   $\checkmark$&   $\checkmark$&$\checkmark$\\
01100   &   $\checkmark$&   $\checkmark$&$\checkmark$\\
01101   &   $\checkmark$&   $\checkmark$&$\checkmark$\\
01102   &            &  $\checkmark$&$\checkmark$\\
01112   &   $\checkmark$&   $\checkmark$&$\checkmark$\\
\hline
\end{tabular}
\quad
\begin{tabular}{c|c|c|c}
\hline
Sequence    &   $S_5$       &$R_5$      &$A_5$\\
\hline\hline
01111&$\checkmark$& $\checkmark$&$\checkmark$\\
01110&$\checkmark$& $\checkmark$&$\checkmark$\\
01120&$\checkmark$& $\checkmark$&$\checkmark$\\
01121&$\checkmark$& $\checkmark$&$\checkmark$\\
01122&$\checkmark$& $\checkmark$&$\checkmark$\\
01123&$\checkmark$& $\checkmark$&$\checkmark$\\
01023&        & $\checkmark$&$\checkmark$\\
01022&        & $\checkmark$&$\checkmark$\\
01021&        & $\checkmark$&$\checkmark$\\
01020&        & $\checkmark$&$\checkmark$\\
01010&$\checkmark$& $\checkmark$&$\checkmark$\\
01011&$\checkmark$& $\checkmark$&$\checkmark$\\
01012&$\checkmark$& $\checkmark$&$\checkmark$\\
01013&        &         &$\checkmark$\\
01002&        & $\checkmark$&$\checkmark$\\
01001&$\checkmark$& $\checkmark$&$\checkmark$\\
01000&$\checkmark$& $\checkmark$&$\checkmark$\\
    &         &          &       \\
\hline
\end{tabular}
\caption{The sets $S_5$, $R_5$, and $A_5$  listed in ${\prec}$ order.}
\label{tab:example_A5}
\end{table}

\begin{table}[!h]
\centering
\begin{tabular}{c|c|c|c}
\hline
Sequence    &   $S_5$       &$R_5$      &$A_5$\\
\hline\hline
01234&  $\checkmark$&   $\checkmark$&$\checkmark$\\
01233   &   $\checkmark$&   $\checkmark$&$\checkmark$\\
01023   &           &   $\checkmark$&$\checkmark$\\
00123   &   $\checkmark$&   $\checkmark$&$\checkmark$\\
01123   &   $\checkmark$&   $\checkmark$&$\checkmark$\\
01223   &   $\checkmark$&   $\checkmark$&$\checkmark$\\
01213   &           &   $\checkmark$&$\checkmark$\\
01013   &           &           &$\checkmark$\\
01203   &           &   $\checkmark$&$\checkmark$\\
01202   &           &   $\checkmark$&$\checkmark$\\
01102   &           &   $\checkmark$&$\checkmark$\\
00102   &           &   $\checkmark$&$\checkmark$\\
01002&          &   $\checkmark$&$\checkmark$\\
01012   &   $\checkmark$&   $\checkmark$&$\checkmark$\\
00012   &   $\checkmark$&   $\checkmark$&$\checkmark$\\
00112   &   $\checkmark$&   $\checkmark$&$\checkmark$\\
01112   &   $\checkmark$&   $\checkmark$&$\checkmark$\\
01212   &   $\checkmark$&   $\checkmark$&$\checkmark$\\
\hline
\end{tabular}
\quad
\begin{tabular}{c|c|c|c}
\hline
Sequence    &   $S_5$       &$R_5$      &$A_5$\\
\hline\hline
01222   &   $\checkmark$&   $\checkmark$&$\checkmark$\\
01122   &   $\checkmark$&   $\checkmark$&$\checkmark$\\
00122   &   $\checkmark$&   $\checkmark$&$\checkmark$\\
01022   &           &   $\checkmark$&$\checkmark$\\
01232   &   $\checkmark$&   $\checkmark$&$\checkmark$\\
01231   &   $\checkmark$&   $\checkmark$&$\checkmark$\\
01021   &           &   $\checkmark$&$\checkmark$\\
00121   &   $\checkmark$&   $\checkmark$&$\checkmark$\\
01121   &   $\checkmark$&   $\checkmark$&$\checkmark$\\
01221   &    $\checkmark$&  $\checkmark$&$\checkmark$\\
01211   &    $\checkmark$&  $\checkmark$&$\checkmark$\\
01111   &   $\checkmark$&   $\checkmark$&$\checkmark$\\
00111   &   $\checkmark$&   $\checkmark$&$\checkmark$\\
00011   &   $\checkmark$&   $\checkmark$&$\checkmark$\\
01011   &   $\checkmark$&   $\checkmark$&$\checkmark$\\
01001   &   $\checkmark$&   $\checkmark$&$\checkmark$\\
00001   &   $\checkmark$ &  $\checkmark$&$\checkmark$\\
00101   &   $\checkmark$&   $\checkmark$&$\checkmark$\\
\hline
\end{tabular}
\quad
\begin{tabular}{c|c|c|c}
\hline
Sequence    &   $S_5$       &$R_5$      &$A_5$\\
\hline\hline
01101&$\checkmark$& $\checkmark$&$\checkmark$\\
01201&$\checkmark$& $\checkmark$&$\checkmark$\\
01200&$\checkmark$& $\checkmark$&$\checkmark$\\
01100&$\checkmark$& $\checkmark$&$\checkmark$\\
00100&$\checkmark$& $\checkmark$&$\checkmark$\\
00000&$\checkmark$& $\checkmark$&$\checkmark$\\
01000&$\checkmark$ &    $\checkmark$&$\checkmark$\\
01010&$\checkmark$ &    $\checkmark$&$\checkmark$\\
00010&$\checkmark$ &    $\checkmark$&$\checkmark$\\
00110&$\checkmark$ &    $\checkmark$&$\checkmark$\\
01110&$\checkmark$ &    $\checkmark$&$\checkmark$\\
01210&$\checkmark$& $\checkmark$&$\checkmark$\\
01220&$\checkmark$& $\checkmark$&$\checkmark$\\
01120&$\checkmark$& $\checkmark$&$\checkmark$\\
00120&$\checkmark$ &    $\checkmark$&$\checkmark$\\
01020&        & $\checkmark$&$\checkmark$\\
01230&$\checkmark$& $\checkmark$&$\checkmark$\\
    &         &          &       \\
\hline
\end{tabular}
\caption{The sets $S_5$, $R_5$, and $A_5$ listed in ${\prec}_c$ order.}
\label{tab:example_Atilde5}
\end{table}


\end{document}